\documentclass[3p, authoryear, 10pt]{elsarticle}
\journal{Statistics \& Probability Letters}

\usepackage{amssymb}
\usepackage{amsmath}
\usepackage[colorlinks]{hyperref}
\usepackage{cleveref}
\usepackage{comment}
\usepackage{graphicx}
\usepackage{caption}
\usepackage{xcolor}

\newtheorem{lemma}{Lemma}
\newtheorem{proposition}{Proposition}

\newproof{proof}{Proof}

\newcommand{\A}{\mathcal A}
\newcommand{\dd}{\mathrm d}

\newcommand{\kl}{\mathrm{kl}}
\newcommand{\KL}{\mathrm{KL}}
\newcommand{\E}{\mathbb E}

\newcommand{\R}{\mathbb R}
\newcommand{\X}{\mathcal X}
\newcommand{\Y}{\mathcal Y}

\definecolor{zaffre}{RGB}{0, 20, 168}

\DeclareMathOperator{\card}{card}

\begin{document}

\bibliographystyle{elsarticle-harv} 

\begin{frontmatter}
    \title{Uniform mean estimation for monotonic processes} 

    \author[1]{Eugenio Clerico}
    \ead{eugenio.clerico@gmail.com}

    \author[1]{Hamish E Flynn}
    \ead{hamishedward.flynn@upf.edu}

    \author[2]{Patrick Rebeschini}
    \ead{patrick.rebeschini@stats.ox.ac.uk}

    \affiliation[1]{organization={Universitat Pompeu Fabra},
    city={Barcelona},
    country={Spain}}
    \affiliation[2]{organization={Department of Statistics, University of Oxford},
    city={Oxford},
    country={UK}}

    \begin{abstract} 
        We consider the problem of deriving uniform confidence bands for the mean of a monotonic stochastic process, such as the cumulative distribution function (CDF) of a random variable, based on a sequence of i.i.d.\ observations. Our approach leverages the coin-betting framework, and inherits several favourable characteristics of coin-betting methods. In particular, for each point in the domain of the mean function, we obtain anytime-valid confidence intervals that are numerically tight and adapt to the variance of the observations. To derive uniform confidence bands, we employ a continuous union bound that crucially leverages monotonicity. In the case of CDF estimation, we also exploit the fact that the empirical CDF is piece-wise constant to obtain simple confidence bands that can be easily computed. In simulations, we find that our confidence bands for the CDF achieve state-of-the-art performance.
    \end{abstract}
\end{frontmatter}

\section{Introduction}

We study the problem of \emph{uniformly} estimating the mean of a \emph{monotonic} stochastic process, namely an almost surely monotonic random function. As a motivating example, consider the problem of estimating the \emph{cumulative distribution function} (CDF), $F(y) = \mathbb{P}(X \leq y)$, of a random variable $X$. Given an i.i.d.~sample $(X_1, \dots, X_T)$, a natural estimator is the \emph{empirical} CDF, $\hat{F}_T(y) = \frac{1}{T}\sum_{t=1}^{T}\mathbb{I}\{X_t \leq y\}$, which has mean $F(y)$ and is non-decreasing in $y$. Monotonic risk functions, such as the risk functions used for calibrating black-box models \citep{bates2021distribution, nguyen2024data}, are another important application.

A vast literature in probability and statistics has focused on CDF estimation. Classical results include the Glivenko-Cantelli theorem \citep{glivenko1933sulla, cantelli1933sulla} and the Dvoretzky-Kiefer-Wolfowitz (DKW) inequality \citep{dvoretzky1956asymptotic}. The Glivenko-Cantelli theorem states that the empirical CDF converges uniformly to the CDF, while the DKW inequality provides a non-asymptotic uniform bound on the absolute difference between the CDF and the empirical CDF. Despite the fact that the constant in the DKW inequality is sharp \citep{massart1990tight}, this bound is often not entirely satisfactory for several reasons. First, the DKW inequality is not \emph{adaptive} to the variance of the empirical CDF, which implies it is typically loose at points where the CDF is close to 0 or 1. Second, the DKW inequality is not \emph{anytime-valid}, in the sense that it only holds for a fixed sample size and does not allow for optional continuation.

In this work, we develop a new method for uniformly estimating the mean of a monotonic bounded stochastic process. Our approach is based on the \emph{coin-betting} framework \citep{orabona2023tight, waudbysmith23estimating}, which is part of a broader line of work on using \emph{gambling} algorithms to derive concentration inequalities \citep{shafer2019game}, and is known to provide numerically tight, anytime-valid and variance-adaptive confidence intervals for bounded random variables. To obtain uniform confidence bands (namely families of confidence intervals holding simultaneously on the whole domain, as in \Cref{fig:3cdfs}), we use a refined continuous union bound of PAC-Bayesian inspiration \citep{alquier2024user, jang2023tighter}, which fundamentally leverages the monotonicity of the process. 

\paragraph{Related work} 
Recent work focuses on improving some of the classical results for CDF estimation. \citet{bartl2023variance} developed a variance-adaptive refinement of the DKW inequality. However, it is not anytime-valid and contains unspecified absolute constants, which makes it hard to implement and test in practice. An approach closer to ours is that of \citet{howard2022sequential}, who obtain anytime-valid and variance-adaptive confidence bands for the CDF using e-values. Their work takes a different strategy to achieve uniformly-valid confidence bands, relying on chaining rather than a PAC-Bayesian approach. Although their results have the advantage of yielding explicit bounds, we find our approach to be more intuitive and simple, as well as often leading to numerically tighter confidence bands (see \Cref{fig:comparison} and the discussion in \Cref{sec:toy}).

We remark that CDF estimation can be framed as a \emph{multiple testing} problem, with a continuum of hypotheses \citep{goldman2018comparing}. From this perspective, our method can be seen as a multiple testing procedure that controls the family-wise error rate (FWER). As a matter of fact, our method makes implicit use of tools from the framework of hypothesis testing with e-values \citep{wasserman2020universal, shafer2021testing, vovk2021evalues, grunwald2024safe, ramdas2024hypothesis}, such as e-processes \citep{ramdas2022admissible} and e-merging functions \citep{vovk2021evalues}. Standard multiple testing procedures based on p-values typically require strong assumptions on the dependence of the p-values used for different hypotheses (independence is often necessary), while their e-value based counterparts are often more flexible due to the merging properties of e-values. For this reason, several recent works have considered multiple testing via e-values. However, most of the current studies focus on controlling the (weaker) false discovery rate \citep{wang2022false, xu2023online, xu2024powerful}, rather than the FWER. The e-Bonferroni procedure \citep{hartog2025family} provides a FWER guarantee, but it uses a simple union bound argument that does not allow for uncountably many hypotheses.

The coin-betting framework has produced some of the tightest known confidence sequences for the mean of bounded random variables \citep{orabona2023tight, waudbysmith23estimating}. This approach can be seen as a particular instance of \emph{algorithmic mean estimation} based on e-values \citep{ramdas2023game}. Recently, \cite{clerico2024optimality} showed that the coin-betting framework is in fact \emph{optimal} among e-value based approaches. The idea of combining coin-betting with continuous union bounds (a.k.a.~PAC-Bayes) has already been proposed in the literature on generalisation bounds \citep{jang2023tighter}. However, we are not aware of any previous works that leverage monotonicity in order to \emph{derandomise} PAC-Bayesian bounds and convert them into uniformly-valid confidence bands.
\paragraph{Notation} We denote sequences as $(s_t)_{t\geq T_0}$, with $t$ an integer index and $T_0$ the initial index (typically $0$ or $1$). We use $\mathbb P$ to express probability with respect to all the randomness involved. For instance, if $(X_t)_{t\geq 1}$ is a sequence of i.i.d.\ draws from a probability law $P$, we  write $\mathbb P(X_t\geq1/2\,,\,\forall t\geq 1)$, with obvious meaning. For integers $n < m$, $[n :m] = \{n, n+1, \dots, m -1, m\}$ denotes the set of all integers between (and including) $n$ and $m$. $\log$ denotes the natural logarithm.

\begin{figure}[t!]
    \centering
    \includegraphics[width=\textwidth]{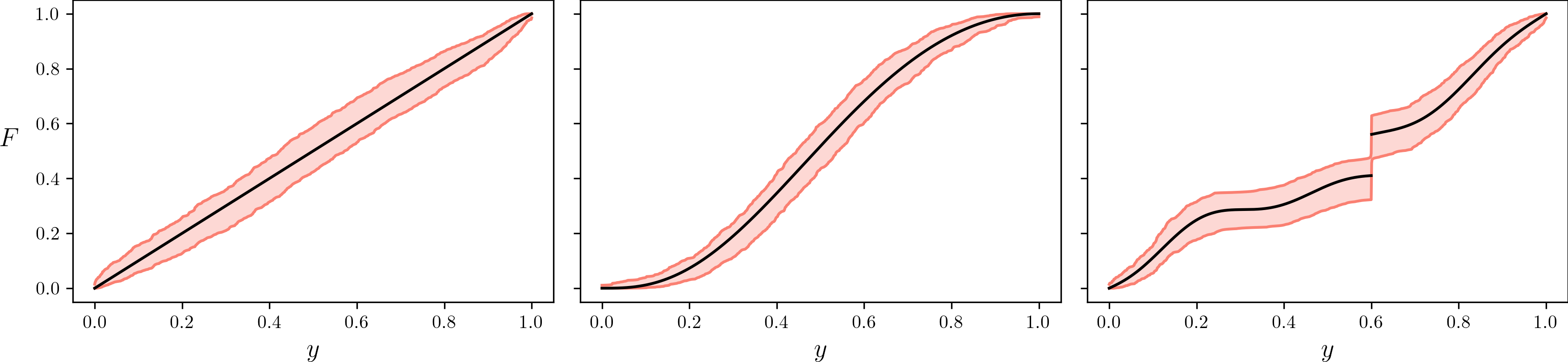}
    \caption{CDF confidence bands from $T=1000$ i.i.d.~observations sampled from three different distributions supported in $[0,1]$. Our method to produce the highlighted bands ensures that they contain the whole black line (the true CDF $F$) with probability at least $0.95$ on the random datasets used to generate them. From left to right (in black): $F(y)=y$, $F(y)=\sin(\pi\sqrt{y}/2)^6$, and $F(y)=(3 + \mathbb{I}\{y\geq0.6\})/4 + \sin(2\pi y^{0.9})^3/10$. The confidence bands reported are those discussed in \Cref{sec:cdf}, and the code used for the numerical evaluation of the binary relative entropy is from \cite{clerico2022conditionally}.}
    \label{fig:3cdfs}
\end{figure}
\section{Setting}

We observe a sequence $(X_t)_{t\geq 1}$ of i.i.d.\ realisations from a Borel probability law $P$ on $\X \subset \R^d$. We are given a Borel function $f: \X \times \Y \to [0,1]$,\footnote{We note that although for simplicity we only consider the case of a map $f$ valued in $[0,1]$, our results are more general. Indeed, if $f$ takes value in a bounded interval $I$, we can always consider an increasing affine injection $\iota:I\to[0,1]$. Then, any confidence set for $\iota\circ f$ translates immediately into a confidence set for $f$.} where $\Y$ is an interval in $\R$. We assume that, for each $x \in \X$, the map $y \mapsto f(x,y)$ is non-decreasing. Our goal is to construct a sequence of confidence bands for the mean function $F: \Y \to [0,1]$, defined as $F(y) = \E_P[f(X,y)]$, which is updated after each $X_t$ in the sequence is observed. In particular, we wish to construct confidence bands that are empirically tight and hold uniformly at all times $t \geq 1$ and all points $y \in \Y$ simultaneously. This unicorn coverage condition is satisfied by any \emph{uniform confidence sequence} for $F$. For any level $\delta \in (0, 1]$, a $(1-\delta)$-uniform confidence sequence for $F$ is a collection of intervals $([L_t(y), U_t(y)])_{t\geq 1,y\in\Y}$, all contained in $[0,1]$, that satisfies the coverage condition
\begin{equation*}
\mathbb P\Big(\forall t \geq 1\,,\,\forall y \in \Y\,,\;F(y) \in [L_t(y), U_t(y)]\Big) \geq 1 - \delta\,.
\end{equation*}
In words, the coverage condition states that with high probability, the value $F(y)$ lies in the interval $[L_t(y), U_t(y)]$ for all times $t \geq 1$ and all points $y \in \Y$ simultaneously. Each sequence $([L_t(y), U_t(y)])_{t \geq 1}$ should of course be adapted to the natural filtration induced by the sequence $(X_t)_{t\geq 1}$. Ideally, we would like the width of the interval $[L_t(y), U_t(y)]$ to scale with the variance of $f(X,y)$, as opposed to the range of values that $f(X,y)$ can take. This variance can substantially change for different values of $y$, so we aim for a confidence band whose width can vary as a function of $y$. 

\section{Method}

We briefly explain how to derive concentration inequalities for bounded random variables via coin-betting \citep{orabona2023tight, waudbysmith23estimating}, and then describe our method for obtaining uniform concentration inequalities for monotonic bounded processes.

\paragraph{Coin-betting} In the coin-betting game, a gambler bets on the outcomes of continuous coins $c_t \in [-a, b]$, where $a \geq 0$, $b \geq 0$, and at least one of these two values is strictly positive. In each round $t$, the gambler bets a proportion $\alpha_t \in [-1/b, 1/a]$\footnote{If $a$ or $b$ is $0$, we adopt the convention $1/0=+\infty$ and allow $\alpha_t$ to be infinite.} of their current wealth $M_{t-1}$, observes $c_t$, and then receives a reward $\alpha_tc_tM_{t-1}$. One can interpret $\mathrm{sgn}(\alpha_t)$ as a bet on the sign of $c_t$ and $|\alpha_t|$ as the size of the bet. The constraint $\alpha_t \in [-1/b, 1/a]$ ensures that the wealth $M_t$ is always non-negative. If the initial wealth of the gambler is one, then the wealth after $T$ rounds is $M_T = \prod_{t=1}^{T}(1 + \alpha_tc_t)$.\footnote{\label{foot:0inf}When necessary, we adopt the convention $0\times\infty = 0$.} The gambler's aim is to accumulate wealth as quickly as possible, and is formalised using the notion of regret. The regret of the gambler is the ratio between the wealth of the best constant betting fraction and the wealth of the gambler. More concretely, the regret of the gambler after $T$ rounds is defined as
\begin{equation*}
R_T = \frac{\sup_{\alpha}\prod_{t=1}^{T}(1 + \alpha c_t)}{\prod_{t=1}^{T}(1 + \alpha_t c_t)}\,.\footnote{We set $R_T=1$ if both numerator and denominator are infinity or if both are null.}
\end{equation*}
If the gambler chooses each $\alpha_t$ according to the Krichevsky-Trofimov (KT) mixture forecaster\footnote{\label{foot:KT}The KT mixture forecaster is classically defined for the investment problem in a binary market. However, it can be defined for the coin-betting problem as $\alpha_t = -\frac{1}{b} + \left(\frac{1}{a}+\frac{1}{b}\right)p_t$, where $$p_t = \frac{\int_0^1 p\left(\prod_{i=1}^{t-1}\left(p\left(1+\frac{c_i}{a}\right) + (1-p)\left(1-\frac{c_i}{b}\right)\right)\right)\frac{\dd p}{\sqrt{p(1-p)}}}{\int_0^1 \left(\prod_{i=1}^{t-1}\left(p\left(1+\frac{c_i}{a}\right) + (1-p)\left(1-\frac{c_i}{b}\right)\right)\right)\frac{\dd p}{\sqrt{p(1-p)}}}\,,$$ if $a$ and $b$ are both non-zero (see the discussion in Section II.C of \citealp{orabona2023tight}). If $a=0$ we let $\alpha_t=+\infty$, and if $b=0$ we let $\alpha_t = -\infty$.}, then for all sequences of coins in $[-1,1]$, the regret is upper bounded as
\begin{equation}
\label{eq:regret}R_T^\textrm{KT}\leq 2\sqrt{T}\,.
\end{equation} (see chapters 9 and 10 in \citealp{cesabianchi2006prediction}, in particular the bound of Theorem 9.4 therein).

\paragraph{Concentration via coin-betting} 
For each $y \in \Y$, we can construct a coin-betting game using the observations $(f(X_t, y) - F(y))_{t \geq 1}$. Specifically, let $c_t(y) = f(X_t, y) - F(y)$. For each $y \in \Y$, $(c_t(y))_{t \geq 1}$ is a sequence of i.i.d.\ random variables taking values in $[-F(y), 1 - F(y)]$. For $\mu \in [0,1]$, we define $\mathcal{A}_{\mu} = [-\frac1{1-\mu}, \frac1{\mu}]$. For each $y \in \Y$, we choose a betting strategy $(\alpha_t(y))_{t \geq 1}$, which is a predictable (w.r.t $(X_t)_{t \geq 1}$) sequence of random variables, each taking values in $\mathcal{A}_{F(y)}$. We require that the mapping $y\mapsto \alpha_t(y)$ is Borel, for all $t$. For each $y \in \Y$, the wealth $M_T(y)$ and the regret $R_T(y)$ of the betting strategy after $T$ rounds are
\begin{equation*}
M_T(y) = \prod_{t=1}^T \Big(1+\alpha_{t}(y)\big(f(X_t, y)-F(y)\big)\Big)\,, \quad R_T(y) = \frac{\sup_{\alpha \in \mathcal{A}_{F(y)}}\prod_{t=1}^{T}\Big(1+\alpha\big(f(X_t, y)-F(y)\big)\Big)}{\prod_{t=1}^{T}\Big(1+\alpha_{t}(y)\big(f(X_t, y)-F(y)\big)\Big)}\,.
\end{equation*}
Since the coins are now zero-mean i.i.d.\ random variables, for any fixed $y$ it is straightforward to check that $M_T(y)$ is a non-negative martingale with initial value 1. From Ville's inequality, for any $\delta \in (0, 1]$,
\begin{equation}\label{eq:ville}
\mathbb{P}\Big(\forall T \geq 1\,,\;\log(M_T(y)) \leq \log(1/\delta)\Big) \geq 1 - \delta\,.
\end{equation}

For any given betting strategy, \eqref{eq:ville} could be used to construct a confidence sequence for $F(y)$. However, to use this confidence sequence, we would have to run the betting strategy, which may be expensive. Instead, \citet{orabona2023tight} introduce the quantity\footnote{If necessary, we adopt the convention $\log 0 + \log(+\infty) = -\infty$, cf.\ \cref{foot:0inf}.}
\begin{equation*}
\psi_T(y, \mu) = \sup_{\alpha \in \mathcal{A}_{\mu}}\left\{\sum_{t=1}^{T}\log\big(1 + \alpha(f(X_t, y) - \mu)\big)\right\}\,.
\end{equation*}
From the definition of the regret $R_T(y)$, we notice that $\psi_T(y, F(y)) = \log(R_T(y)) + \log(M_T(y))$. Thus, if we set each $\alpha_t(y)$ using the KT mixture forecaster\footnote{Such a choice is compatible with the measurability requirement for $y\mapsto\alpha_t(y)$. This can easily be deduced from the explicit form given in \cref{foot:KT}.}, \eqref{eq:regret} yields
\begin{equation}
\mathbb{P}\Big(\forall T \geq 1\,,\;\psi_T(y, F(y)) \leq \log\tfrac{2\sqrt{T}}{\delta}\Big) \geq 1 - \delta\,.\label{eq:psi_ineq}
\end{equation}

Since $\psi_T(y, F(y))$ does not depend on the betting strategy used, this inequality can be leveraged to construct a confidence sequence for $F(y)$ that does not depend on $\alpha_t(y)$. For $C\geq 0$ (possibly infinite), we define the upper, $\psi_{T,+}^{-1}$, and lower, $\psi_{T,-}^{-1}$, inverses (w.r.t.\ $\mu$) of $\psi_T(y, \mu)$ as
\begin{equation}\label{eq:psinv}
\psi_{T, +}^{-1}(y, C) = \sup\big\{\mu\in[0,1]\,:\,\psi_{T}(y, \mu)\leq C\big\}\,, \qquad \psi_{T, -}^{-1}(y, C) = \inf\big\{\mu\in[0,1]\,:\,\psi_{T}(y, \mu)\leq C\big\}\,.
\end{equation}
Using these definitions,
\begin{equation*}
\mathbb{P}\Big(\forall T \geq 1\,,\;F(y) \in \big[\psi_{T, -}^{-1}(y, \log\tfrac{2\sqrt{T}}{\delta})\,, \psi_{T, +}^{-1}(y, \log\tfrac{2\sqrt{T}}{\delta})\big]\Big) \geq 1 - \delta\,.
\end{equation*}

To obtain a uniform confidence sequence for $F(y)$, we need a stronger version of the inequality in \eqref{eq:psi_ineq}, holding simultaneously for all $y \in \Y$. Recently, \citet{jang2023tighter} leveraged the PAC-Bayesian framework \citep{alquier2024user} to obtain bounds for $F$ averaged on $y$. Let $\pi$ be any fixed probability measure on $\Y$ (whose choice does not depend on the observations $(X_t)_{t\geq 1}$). Using Donsker and Varadhan's variational formula for the KL divergence \citep{donsker1976asymptotic}, for every $\rho$ and $\pi$,
\begin{equation}\label{eq:bound}
    \int_{\Y}\psi_T(y, F(y))\dd\rho(y) \leq \KL(\rho|\pi) + \log\int_\Y M_T(y)\dd\pi(y) +\log(2\sqrt T)\,.
\end{equation}
We remark that the above integrals make sense as both $y\mapsto M_T(y)$ and $y\mapsto\psi_T(y,F(y))$ are non-negative Borel measurable functions.\footnote{The non-negativity is trivial. To show the measurability, let $g_\alpha(y) = \sum_{t=1}^{T}\log(1+\alpha(f(X_t,y)-F(y)))$ if $\alpha\in\A_{F(y)}$, and $-\infty$ otherwise. Let $G(y) = \sup_{\alpha\in\R}g_\alpha(y)$. Then, since $0\in\A_{F(y)}$ for any $y$ and $g_0=0$, we have that $\psi_T(y,F(y)) = G(y)$. Now, the measurability of $F$ and $f(X_t,\cdot)$ implies that each $g_\alpha$ is measurable. Moreover, using continuity, and from the fact that $\A_{F(y)}$ is an interval containing $0$, one can easily verify that $G(y)= \sup_{\alpha\in\mathbb Q\cup\{\pm\infty\}}g_\alpha(y)$, and hence $G$ is measurable as it is the supremum of countably many Borel measurable functions.} Since $\int_{\Y}M_T(y)\dd \pi(y)$ is also a non-negative martingale with initial value 1,\footnote{Note that this can be seen as an instance of merging of e-processes (see \citealp{vovk2021evalues, ramdas2024hypothesis}).} it too is bounded by $1/\delta$ with probability at least $1 - \delta$ by Ville's inequality. Therefore,
\begin{equation*}
\mathbb{P}\left(\forall T \geq 1\,, \;\sup_{\rho}\left\{\int_{\Y}\psi_T(y, F(y))\dd\rho(y) - \KL(\rho|\pi)\right\} \leq \log\tfrac{2\sqrt{T}}{\delta}\right) \geq 1 - \delta\,.
\end{equation*}

\paragraph{Leveraging monotonicity} 
So far we have a concentration inequality that holds uniformly on the set of probability measures on $\Y$, but is vacuous when $\rho$ is a point mass. To obtain an inequality that holds uniformly on $\Y$, we exploit the fact that the monotonicity of $f$ implies that the upper and lower inverse of $\psi_T$ are monotonic in their arguments.

\begin{lemma}\label{lemma:monul}
    For any fixed $C\geq 0$ and any $T$, both $y\mapsto\psi_{T,+}^{-1}(y,C)$ and $y\mapsto\psi_{T,-}^{-1}(y,C)$ are non-decreasing. For any fixed $y\in\Y$ and any $T$, $C\mapsto\psi_{T,+}^{-1}(y,C)$ is non-decreasing and $C\mapsto\psi_{T,-}^{-1}(y,C)$ is non-increasing.
\end{lemma}
\begin{proof}
    The monotonicity in $y$ follows from \Cref{lemma:monotonic} and the fact that $f$ is non-decreasing in $y$. The monotonicity in $C$ is obvious.\qed
\end{proof}

We now state our main result for the case $\Y=[0,1]$, though it can readily be applied to general intervals. Indeed, one can always reparameterise $\Y$ via an increasing injection $\iota:\Y\to[0,1]$, whose image is an interval with closure $[0,1]$. If $0$ or $1$ are not in $\iota(\Y)$, one can define an extension $\tilde F$ (on $[0,1]$) of $F\circ\iota$, by taking the necessary limits ($y\to0$ and/or $y\to1$) of $F\circ\iota$ (which exist due to monotonicity). One can then translate a uniform guarantee for $\tilde F$ into a uniform guarantee for $F$.

\begin{proposition}\label{pro:psi_cs}
    Let $\Y=[0,1]$. With probability at least $1-\delta$, uniformly on $T\geq1$ and on $y_0\in\Y$, we have 
    $$F(y_0)\in\left[\sup_{y_-\in[0,y_0]}\psi_{T,-}^{-1}\Big(y_-, \log\tfrac{2\sqrt T}{(y_0-y_-)\delta}\Big),\inf_{y_+\in[y_0,1]}\psi_{T,+}^{-1}\Big(y_+,\log\tfrac{2\sqrt T}{(y_+-y_0)\delta}\Big)\right]\,.$$
\end{proposition}
\begin{proof}   
    For $y_1,y_2\in [0,1]$, we let $\rho_{y_1,y_2}$ denote the uniform distribution on $[y_1,y_2]$ (a Dirac mass if $y_1=y_2$). In what follows we adopt the convention $1/0=+\infty$ and $\log0=-\infty$.
    
    Fix $y_0\in\Y$ and $y_+\in[y_0,1]$. There exists $\hat y\in[y_0,y_+]$ such that $\psi_T(\hat y, F(\hat y))\leq\int_{0}^{1}\psi_T(y, F(y))\dd\rho_{y_0,y_+}(y)$. When using the KT forecaster, since $\KL(\rho_{y_0,y_+}|\rho_{0,1}) = -\log(y_+-y_0)$, from \eqref{eq:psinv} and \eqref{eq:bound} we obtain
    $$F(\hat y) \leq \psi_{T,+}^{-1}\left(\hat y, \log\tfrac{2\sqrt T}{y_+-y_0} + \textstyle\int_0^1M_T(y)\dd y\right)\,.$$
    From the monotonicity of $F$ and \Cref{lemma:monul}, it follows that
    $$F(y_0) \leq \psi_{T,+}^{-1}\left(y_+, \log\tfrac{2\sqrt T}{y_+-y_0} + \textstyle\int_0^1M_T(y)\dd y\right)\,.$$
    Similar arguments yield that, for any $y_-\in[0,y_0]$, we have
    $$F(y_0) \geq \psi_{T,-}^{-1}\left(y_-, \log\tfrac{2\sqrt T}{y_0-y_-} + \textstyle\int_0^1M_T(y)\dd y\right)\,.$$
    Thus, for all $y_0\in\Y$,
    $$\sup_{y_-\in[0,y_0]}\psi_{T,-}^{-1}\left(y_-, \log\tfrac{2\sqrt T}{y_0-y_-} + \textstyle\int_0^1M_T(y)\dd y\right)\leq F(y_0)\leq \inf_{y_+\in[y_0,1]}\psi_{T,+}^{-1}\left(y_+, \log\tfrac{2\sqrt T}{y_+-y_0} + \textstyle\int_0^1M_T(y)\dd y\right)\,.$$
    From \Cref{lemma:monul}, the above upper and lower bounds are respectively non-decreasing and non-increasing in $\int_0^1M_T(y)\dd y$. The conclusion follows from Ville's inequality as $\int_0^1M_T(y)\dd y$ is a non-negative martingale.\qed
\end{proof}
It is worth mentioning that $\psi_T$ and its lower and upper inverses can be computed efficiently, thanks to the quasi-convexity of $\psi_T$ in $y$. We refer to the related discussion in \cite{orabona2023tight} for more details.

\section{Relaxations}

The intervals in \Cref{pro:psi_cs} do not lend themselves readily to interpretation. By lower bounding $\psi(y, \mu)$, we can obtain more explicit confidence intervals.

\paragraph{Binary relative entropy} 

Confidence intervals for bounded random variables are often expressed in terms of the KL divergence between Bernoulli distributions, which we refer to as the \emph{binary relative entropy}. For $p, q \in (0, 1)$, we define the binary relative entropy as
\begin{equation*}
\kl(p,q) = p\log\tfrac pq + (1-p)\log\tfrac{1-p}{1-q}\,.
\end{equation*}
We also let $\kl(0,q) = \log\frac1{1-q}$ for $q\in[0,1)$, and $\kl(1,q) = \log\frac1q$ for $q\in(0,1]$. If $p\in(0,1]$ then $\kl(p,0) = +\infty$, and $\kl(p,1)=+\infty$ if $p\in[0,1)$. For $C \geq 0$, we define the upper and lower inverses (w.r.t.\ $q$) of $\kl$ as
\begin{equation*}
\kl^{-1}_+(p, C) = \sup\big\{q\in[0,1]\,:\,\kl(p,q)\leq C\big\}\,, \qquad \kl^{-1}_-(p, C) = \inf\big\{q\in[0,1]\,:\,\kl(p,q)\leq C\big\}\,.
\end{equation*}
Proposition 1 in \cite{orabona2023tight} states that, for every $\mu \in [0, 1]$, \begin{equation}\label{eq:psikl}\psi_T(y, \mu) \geq T\,\kl(\hat{F}_T(y), \mu)\,,\end{equation} where $\hat{F}_T(y) = \frac{1}{T}\sum_{t=1}^{T}f(X_t, y)$. Moreover, this holds with equality whenever $(f(X_t, y))_{t \geq 1}$ is a binary sequence. From \eqref{eq:psikl}, it follows that the inverses of $\psi_T$ and $\kl$ satisfy $\psi^{-1}_{T,+}(y,C) \leq \kl^{-1}_+(\hat{F}_T(y),C/T)$ and $\psi^{-1}_{T,-}(y,C) \geq \kl^{-1}_-(\hat{F}_T(y),C/T)$. Therefore, from \Cref{pro:psi_cs}, with probability at least $1 - \delta$, uniformly for all $T \geq 1$ and all $y \in [0, 1]$,
\begin{equation}
F(y) \in \left[\sup_{y_- \in [0, y]}\kl_-^{-1}\left(\hat{F}_T(y), \tfrac{1}{T}\log\tfrac{2\sqrt{T}}{(y-y_-)\delta}\right)\,,\, \inf_{y_+ \in [y, 1]}\kl_+^{-1}\left(\hat{F}_T(y), \tfrac{1}{T}\log\tfrac{2\sqrt{T}}{(y-y_-)\delta}\right)\right]\,.\label{eqn:klset}
\end{equation}
This confidence sequence depends only on the sample mean $\hat{F}_T(y)$, so it can be updated without storing the entire sequence $(X_t)_{t \geq 1}$. In general, the intervals in this confidence sequence will be larger than those in \Cref{pro:psi_cs}. However, when each $f(X, y)$ is a Bernoulli random variable, as is the case in CDF estimation (\Cref{sec:cdf}), this confidence sequence is equivalent to the one in \Cref{pro:psi_cs}.

Leveraging Pinsker's inequality, which for Bernoulli random variables reads $\kl(p,q) \geq 2(p-q)^2$, we obtain a further relaxation, with a form evoking the classical DKW inequality. With probability at least $1 - \delta$, uniformly on $T$ and $y$, 
\begin{equation}\label{eq:pinsker}F(y)\in\left[\sup_{y_-\in[0,y]}\Big(\hat F_T(y_-) - \sqrt{ \tfrac{1}{2T}\log\tfrac{2\sqrt T}{(y-y_-)\delta}}\Big)\,,\, \inf_{y_+\in[y,1]}\Big(\hat F_T(y_+) + \sqrt{ \tfrac{1}{2T}\log\tfrac{2\sqrt T}{(y_+-y)\delta}}\Big)\right]\,.\end{equation}
Note that the width of this last confidence band is approximately constant w.r.t.\ $y$, and in general can be looser than necessary. However, this is mainly due to the looseness of Pinsker's inequality, and the $\kl$ confidence band can be much tighter.

\paragraph{Variance-dependence} 
As we mentioned in the introduction, our bound is adaptive to the variance of the observations. This is due to the fact that the coin-betting framework yields variance-adaptive concentration inequalities for bounded random variables. More specifically, for any fixed $y$, a simplification\footnote{We use here that $1-\frac{2}{T}\log(2\sqrt T)\geq\frac12$ for $T\geq 10$, and that for positive $A$ and $B$ we have $\sqrt{A+B}\leq\sqrt A+\sqrt B$.} of what is discussed in the proof of Theorem 3 of \cite{orabona2023tight} yields that $\psi_{T, +}^{-1}(y,C) - \hat F_T(y)\leq D_T(y,C)$ and $\hat F_T(y)-\psi_{T, -}^{-1}(y,C) \leq D_T(y,C)$, for any $T\geq 10$ and $C\geq \log(2\sqrt T)$. Here,
$$D_T(y, C) = \tfrac{8C}{3T} + \sqrt{\tfrac{4C\hat V_T(y)}{T}}\,,$$
where $\hat V_T(y) = \frac{1}{T}\sum_{t=1}^T (f(X_t, y)-\hat F_T(y))$ is the empirical variance. Therefore, a relaxation of \Cref{pro:psi_cs} yields the variance adaptive $(1-\delta)$-confidence band (uniform for $T\geq 10$ and $y\in[0,1]$)
$$F_T(y) \in \left[\sup_{y_-\in[0,y]}\left(\hat F_T(y_-) - D_T\big(y_-, \log\tfrac{2\sqrt T}{(y-y_-)\delta}\big)\right),\inf_{y_+\in[y,1]}\left(\hat F_T(y_+) + D_T\big(y_+, \log\tfrac{2\sqrt T}{(y_+-y)\delta}\big)\right)\right]\,.$$

\paragraph{Asymptotic rates} The above relaxations, especially \eqref{eq:pinsker}, suggest an asymptotic rate of $O(\sqrt{T^{-1}\log T})$ for the width of our confidence bands. However, a formal proof of this statement in the general case is non trivial, due to the fact that our upper and lower bounds for $F(y)$ depend on $\hat F_T(y_+)$ and $\hat F_T(y_-)$, rather than on $\hat F_T(y)$. We leave a rigorous treatment of this problem for future work, and instead give a heuristic argument for the typical width when the distribution $P$ generating the sequence of observations has bounded density, and the function $F:[0,1]\to[0,1]$ is Lipschitz. For large $T$, in a sample of $T$ data-points we expect that no more than $O(\sqrt T)$ lie in $[y, y+T^{-1/2}]$ for any fixed $y\in(0,1)$. In particular, we will have $\hat F_T(y+T^{-1/2}) = F(y) + O(T^{-1/2})$. Then, picking $y_\pm = y\pm\Omega(T^{-1/2})$ (for $T$ large enough so that these values lie in $[0,1]$) we obtain from \eqref{eq:pinsker} that $|F(y)-\hat F(y)|\leq O(\sqrt{T^{-1}\log T})$ with high probability. We refer to \Cref{sec:conc} for some further remarks on asymptotic rates, in view of the law of the iterated logarithm.

\section{Application to CDF estimation}
\label{sec:cdf}
We return to the CDF estimation problem described in the introduction, where $f(x,y) = \mathbb{I}\{x \leq y\}$. For simplicity, we assume that $P$ is supported in $[0,1]$.\footnote{For the general case, see the discussion preceding the statement of \Cref{pro:psi_cs}.} Since each $\mathbb{I}\{X_t \leq y\}$ is a Bernoulli random variable, we use the binary relative entropy confidence sequence in \eqref{eqn:klset}. At first glance, it is not obvious how to efficiently compute the confidence intervals in \eqref{eqn:klset}, as each end point is defined implicitly as the solution of a maximisation or minimisation over infinitely many $y_-$ or $y_+$. However, when $f(x,y) = \mathbb{I}\{x \leq y\}$, we can exploit that the mapping $y \mapsto \hat F_T(y) = \frac{1}{T}\sum_{t=1}^{T}\mathbb{I}\{X_t \leq y\}$ is piece-wise constant. More precisely, let
$$C_T(y) = \card\big\{t\in[1:T]\,:\,X_t\leq y\big\}\,,$$
with $\card$ denoting the cardinality. For any $T$, we use $X^{(T)}_{t}$ to denote the $t$\textsuperscript{th} largest element among the first $T$ observations. Also, let $X^{(T)}_{0} = 0$ and $X^T_{(T+1)}=1$. We notice that $\hat F_T(y) = C_T(y)/T$ is equal to the constant $t/T$ on the interval $[X^{(T)}_{t}, X^{(T)}_{t+1})$. Since $\kl^{-1}_+$ is increasing in both of its arguments, for every interval $[a,b]\subseteq [X^{(T)}_{t}, X^{(T)}_{t+1})$, with $a\geq y$, we have $\inf_{y'\in[a, b]}\kl^{-1}_+\big(\hat F_T(y'),\tfrac{1}{T}\log\tfrac{2\sqrt T}{(y'-y)\delta}\big) = \kl^{-1}_+\big(\tfrac{t}{T},\tfrac{1}{T}\log\tfrac{2\sqrt T}{(b-y)\delta}\big)$. In particular,
$$\inf_{y_+\geq y}\kl^{-1}_+\Big(\hat F_T(y_+),\tfrac{1}{T}\log\tfrac{2\sqrt T}{(y_+-y)\delta}\Big) = \min_{t\in[C_T(y):T]}\kl^{-1}_+\Big(\tfrac{t}{T},\tfrac{1}{T}\log\tfrac{2\sqrt T}{(X^{(T)}_{t+1}-y)\delta}\Big)\,.$$
Similarly, one can show that $$\sup_{y_-\leq y}\kl^{-1}_-\Big(\hat F_T(y_-),\tfrac{1}{T}\log\tfrac{2\sqrt T}{(y-y_-)\delta}\Big) = \max_{t\in[ 0:C_T(y)]}\kl^{-1}_-\Big(\tfrac{t}{T},\tfrac{1}{T}\log\tfrac{2\sqrt T}{(y-X^{(T)}_{t})\delta}\Big)\,.$$
Therefore, the minimum and maximum are always achieved on a finite set of at most $T$ points.

\subsection*{Experimental comparison}\label{sec:toy}
\begin{figure}[t!]
    \centering
    \includegraphics[width=\textwidth]{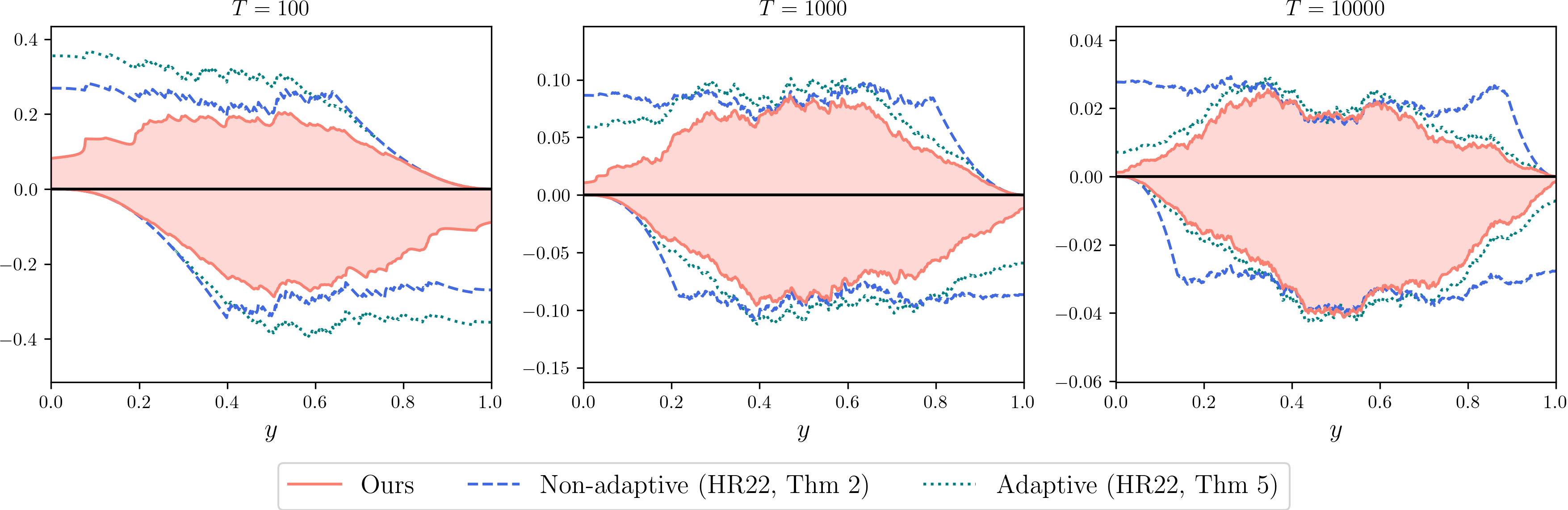}
    \caption{Confidence bands ($\delta=0.05$) for $F(y) = \sin\big(\frac\pi2\sqrt y\big)^6$. The difference between the confidence bands and the true CDF is reported on the vertical axis. The confidence bands via our approach are plotted with a solid line. The dashed lines denote the confidence bands from Theorem 2 in \cite{howard2022sequential}, $\|\hat F_T-F\|_\infty \leq 0.85\sqrt{\tfrac1T\big(\log\log(eT) + 0.8\log(1612/\delta)\big)}$. The dotted lines depict the variance-adaptive implicit confidence bands from Theorem 5 in \cite{howard2022sequential}.}
    \label{fig:comparison}
\end{figure}
We compare our CDF confidence bands against both the variance-adaptive and non-adaptive confidence bands from \citep{howard2022sequential}. To the authors knowledge, this is the only currently available work to provide tight anytime-valid and variance adaptive confidence bands for the CDF. We tested our method for different sample sizes ($T=100$, $1000$, and $10000$), on a toy dataset, made of i.i.d.\ draws from a distribution supported on $[0,1]$, with CDF $$F(y) = \sin\big(\tfrac\pi2\sqrt y\big)^6\,.$$ The results are reported in \Cref{fig:comparison}. Our approach (solid line) consistently outperforms the variance-adaptive method (dotted line) proposed in \cite{howard2022sequential}. Due to the lack of variance-adaptiveness, the non-adaptive confidence band (dashed line) proposed in the same work is much looser when $F(y)$ is close to $0$ and $1$. However, as the dataset gets larger, this last method performs slightly better than ours for values where $F(y)$ is far from $0$ and $1$. This is however expected, as the non-adaptive confidence band has an asymptotic width that matches the law of the iterated logarithm ($\sqrt{T^{-1}\log\log T}$), while ours cannot shrink faster than $\sqrt{T^{-1}\log T}$. We refer to the discussion in \Cref{sec:conc} for how to improve the convergence rate of our approach. It is however worth noticing the remarkable tightness of our approach for the smallest dataset.

\section{Technicalities}
We fix an integer $T\geq 1$. For $z$ and $z'$ in $[0,1]^T$, we write $z\succeq z'$ if $z_t\geq z'_t$ for all $t\in[1:T]$. For $\mu\in[0,1]$, we set $\A_\mu = [-\tfrac{1}{(1-\mu)},\tfrac1\mu]$. We define
$$\Psi(z,\mu) = \sup_{\alpha\in \A_\mu}\frac1T\sum_{t=1}^{T}\log\big(1+\alpha(z_t-\mu)\big)\,.$$
Note that $\Psi(z,\mu)$ is non-negative (as $0\in \A_\mu$). Letting $\mu_z=\frac{1}{T}\sum_{t=1}^{T}z_t$, we have $\Psi(z,\mu_z) = 0$. Indeed, Jensen's inequality yields that $\frac1T\sum_{t=1}^{T}\log\big(1+\alpha(z_t-\mu)\big)\leq 0$ for any $\alpha$, so for $\alpha=0$ the sup is achieved. In particular, $\Psi(z,\mu)$ has a minimum for $\mu=\mu_z$. Moreover, Theorem 4 in \cite{orabona2023tight} shows that it is quasi-convex in $\mu$. In particular, it is non-decreasing in $\mu$, for $\mu\geq\mu_z$, and non-increasing for $\mu\leq\mu_z$.

\begin{lemma}\label{lemma:Psi}
    Fix $z$ and $z'$, such that $z\succeq z'$.  Let $\mu\in [0,1]$. If $\mu\geq\mu_z$, then
    $\Psi(z,\mu)\leq\Psi(z',\mu)$.
    Conversely, if $\mu\leq\mu_{z'}$, we have $\Psi(z,\mu)\geq\Psi(z',\mu)$.
\end{lemma}
\begin{proof}
    We prove the first statement, the second follows similarly. Let $\mu\geq\mu_z$. From Jensen's inequality, $\frac1T\sum_{t=1}^{T}\log\big(1+\alpha(z_t-\mu)\big)\leq \log\left(1+\alpha (\mu_z-\mu)\right)\leq\log 1 = 0$,
    for any $\alpha>0$. In particular, we obtain that
    $\Psi(z,\mu) = \sup_{\alpha\in[-\frac{1}{1-\mu},0]}\frac1T\sum_{t=1}^{T}\log\big(1+\alpha(z_t-\mu)\big)$. If $\alpha\in[-\frac{1}{1-\mu},0]$, $\log\big(1+\alpha(z_t-\mu)\big)\leq\log\big(1+\alpha(z'_t-\mu)\big)$, since $z\succeq z'$. Therefore,
    $$\Psi(z,\mu) = \sup_{\alpha\in[-\frac{1}{1-\mu},0]}\frac1T\sum_{t=1}^{T}\log\big(1+\alpha(z_t-\mu)\big)\leq \sup_{\alpha\in[-\frac{1}{1-\mu},0]}\frac1T\sum_{t=1}^{T}\log\big(1+\alpha(z'_t-\mu)\big)\leq\Psi(z',\mu)\,,$$ and so we conclude.\qed
\end{proof}

Fix $C\geq 0$ (possibly infinite). We define the upper and lower inverses of $\Psi$ as $$\Psi^{-1}_+(z,C) = \sup\big\{\mu\in[0,1]\,:\,\Psi(z,\mu)\leq C\big\}\,,\qquad\Psi^{-1}_-(z,C) = \inf\big\{\mu\in[0,1]\,:\,\Psi(z,\mu)\leq C\big\}\,.$$

\begin{lemma}\label{lemma:monotonic}
    For any fixed $C\geq 0$, the mappings $z\mapsto\Psi^{-1}_+(z,C)$ and $z\mapsto\Psi^{-1}_-(z,C)$ are non-decreasing (with respect to the partial ordering $\succeq$).
\end{lemma}
\begin{proof}
    We prove that $\Psi^{-1}_+(\cdot,C)$ is non-decreasing, as the proof for $\Psi^{-1}_-(\cdot,C)$ is analogous. Fix $z$ and $z'$, with $z\succeq z'$. For any $\mu\geq\mu_z$ we have that $\Psi(z,\mu)\leq\Psi(z',\mu)$ by \Cref{lemma:Psi}. Thus, $$\big\{\mu\geq\mu_z\,:\,\Psi(z,\mu)\leq C\big\} \supseteq \big\{\mu\geq\mu_z\,:\,\Psi(z',\mu)\leq C\big\}\,.$$ As $\{\mu\geq\mu_z\,:\,\Psi(z,\mu)\leq C\}$ is non-empty ($\mu_z$ belongs to it), the conclusion follows.\qed 
\end{proof}

\section{Conclusion}\label{sec:conc}

In this work, we developed a method for uniform mean estimation of monotonic processes. When applied to the problem of CDF estimation, our approach yields CDF confidence bands that have a relatively simple expression, and are amenable to efficient computation. In our experiments, we found that our confidence bands are numerically tight, especially for small sample sizes.

We comment on a limitation of our work. Throughout the paper, we used the fact that the log regret of the KT forecaster is at most $\log(2\sqrt{T})$. However, this means that our confidence bands shrink at best at a rate of $\sqrt{T^{-1}\log T}$ (where the variance is non-zero), rather than the fastest possible rate of $\sqrt{T^{-1}\log\log T}$. To match the rate of the law of the iterated logarithm, one could use the betting algorithm from Section IV of \cite{orabona2023tight}, though this would come at the expense of a more complicated upper bound on $\log(R_T(y))$.

As a final remark, while we have demonstrated empirically that our approach leads to numerically tight anytime-valid confidence bands, we cannot rule out the possibility that other e-value based approaches would be even more effective. We leave open the question of what are the log-optimal \citep{grunwald2024safe, larsson2024numeraire} and admissible \citep{ramdas2022admissible} e-variables in this context, and the related problem of characterising the optimal e-class, in the sense of \cite{clerico2024optimality,clerico2024optimal}.

\paragraph{Acknowledgements} 
We thank Gergely Neu for the discussions that inspired this work, and Amir Reza Asadi for insightful comments. Eugenio and Hamish are funded by the European Research Council (ERC), under the European Union’s Horizon 2020 research and innovation programme (grant agreement 950180). Patrick is funded by UK Research and Innovation (UKRI) under the UK government’s Horizon Europe funding guarantee (grant number EP/Y028333/1).

\bibliography{bib.bib}
\end{document}